\documentclass[11pt]{amsart}
\usepackage{amsthm}
\usepackage{amsmath,amstext,amsthm,amsfonts,amssymb,amsthm}
\usepackage{mathrsfs}
\usepackage{multicol}
\usepackage{latexsym}
\usepackage{color}
\usepackage{graphicx}
\usepackage{epsf}
\usepackage{epsfig}
\usepackage{epic}
\usepackage{graphics}
\usepackage{verbatim}
\usepackage{color}
\usepackage{bm}
\newtheorem{theorem}{Theorem}[section]

\theoremstyle{definition}
\newtheorem{definition}[theorem]{Definition}
\newtheorem{proposition}[theorem]{Proposition}

\theoremstyle{remark}

\numberwithin{equation}{section}

\newcommand{\D}{\mathcal{D}}
\newcommand{\C}{\mathbb{C}}

\newcommand{\B}{\mathcal{B}}

\newcommand{\oqu}{\overline{q}}

\newcommand{\quat}{\mathbb{H}}

\begin{document}
\title[Frame Perturbation]{Perturbation of Continuous Frames  on  Quaternionic Hilbert Spaces}
\author[M. Khokulan, K. Thirulogasanthar]{M. Khokulan$^1$, K. Thirulogasanthar$^2$.}
\address{$^{1}$ Department of Mathematics and Statistics, University of Jaffna, Thirunelveli, Jaffna, Srilanka. }
\address{$^{2}$ Department of Computer Science and Software Engineering, Concordia University, 1455 de Maisonneuve Blvd. West, Montreal, Quebec, H3G 1M8, Canada.}
\email{mkhokulan@gmail.com, santhar@gmail.com.}
\date{\today}
\begin{abstract}
In this note, following the theory of discrete frame perturbations in a complex Hilbert space, we examine perturbation of rank $n$ continuous frame, rank $n$ continuous Bessel family and rank $n$ continuous Riesz family in a non-commutative setting, namely in a right quaternionic Hilbert space.
\end{abstract}
\maketitle
\pagestyle{myheadings}
\section{Introduction}
Frame is a spanning set of vectors which was introduced by Duffin and Schaeffer in 1952 in the study of non-harmonic Fourier series \cite{DU} and then in 1986 a landmark development was given by Daubechies et al. in \cite{D1,D2}.  Since then the frame theory had been widely studied by several authors \cite{D2, GR, Ole}. The study of frames has attracted interest in recent years because of their applications in several areas of Engineering, Applied Mathematics and Mathematical Physics. Many applications of frames have arisen in recent years, for example, Internet coding \cite{STR}, sampling \cite{Eld}, filter bank theory \cite{Bol}, system modeling \cite{Dud}, digital signal processing \cite{GR,Ole} and many more.\\

Perturbation theory plays a significant role in several areas of mathematics. Frame perturbations were first explicitly introduced
by Chris Heil in his Ph.D. thesis \cite{Hei}, and then widely studied by other authors \cite{Chen, Pgc,Ole1,Ole2,Zal}. As far as we know, these perturbation results have not been extended even to the complex continuous frames. In this paper we investigate  certain perturbations  of rank $n$ continuous frames in a right quaternionic Hilbert space, which was introduced in \cite{Kho}, along the lines of the arguments given in  \cite{Chen,Ole3}, where  frame perturbations were studied for complex discrete frames. \\

This article is organized as follows. In section 2, we collect some basic notations and preliminary results about quaternions and frames as needed for the development of the results obtained in this article. In section 3, we present the main results of this article, that is,  perturbations of rank $n$ continuous  frames, rank $n$ continuous Bessel family and rank $n$ continuous Riesz family in right quaternionic Hilbert spaces following their discrete counterparts studied in complex Hilbert spaces. 
\section{Mathematical preliminaries}
We recall few facts about quaternions, quaternionic Hilbert spaces and quaternionic operator properties which may not be very familiar to the reader. For more details on quaternions and quaternionic Hilbert spaces we refer the reader to \cite{Ad}. 
\subsection{Quaternions}
Let $\mathbb{H}$ denote the field of quaternions. Its elements are of the form $q=x_0+x_1i+x_2j+x_3k,~$ where $x_0,x_1,x_2$ and $x_3$ are real numbers, and $i,j,k$ are imaginary units such that $i^2=j^2=k^2=-1$, $ij=-ji=k$, $jk=-kj=i$ and $ki=-ik=j$. The quaternionic conjugate of $q$ is defined to be $\overline{q} = x_0 - x_1i - x_2j - x_3k$. Quaternions do not commute in general. However $q$ and $\oqu$ commute, and quaternions commute with real numbers. $|q|^2=q\oqu=\oqu q$ and $\overline{qp}=\overline{p}~\oqu.$ The quaternion field is measurable and we take a measure $d\mu$ on it. For instant $d\mu$ can be taken as a Radon measure or $d\mu=d\lambda d\omega$, where $d\lambda$ is a Lebesgue measure on $\C$ and $d\omega$  is a Harr measure on $SU(2)$. For details we refer the reader to, for example, \cite{Thi1} (page 12).

\subsection{Right Quaternionic Hilbert Space}
Let $V_{\mathbb{H}}^{R}$ be a linear vector space under right multiplication by quaternionic scalars (again $\mathbb{H}$ standing for the field of quaternions).  For $\phi ,\psi ,\omega\in V_{\mathbb{H}}^{R}$ and $q\in \mathbb{H}$, the inner product
$$\langle\cdot\mid\cdot\rangle:V_{\mathbb{H}}^{R}\times V_{\mathbb{H}}^{R}\longrightarrow \mathbb{H}$$
satisfies the following properties
\begin{enumerate}
	\item[(i)]
	$\overline{\langle\phi \mid \psi \rangle}=\langle \psi \mid\phi \rangle$
	\item[(ii)]
	$\|\phi\|^{2}=\langle\phi \mid\phi \rangle>0$ unless $\phi =0$, a real norm
	\item[(iii)]
	$\langle\phi \mid \psi +\omega\rangle=\langle\phi \mid \psi \rangle+\langle\phi \mid \omega\rangle$
	\item[(iv)]
	$\langle\phi \mid \psi q\rangle=\langle\phi \mid \psi \rangle q$
	\item[(v)]
	$\langle\phi q\mid \psi \rangle=\overline{q}\langle\phi \mid \psi \rangle$
\end{enumerate}
where $\overline{q}$ stands for the quaternionic conjugate. We assume that the
space $V_{\mathbb{H}}^{R}$ is complete under the norm given above. Then,  together with $\langle\cdot\mid\cdot\rangle$ this defines a right quaternionic Hilbert space, which we shall assume to be separable. Quaternionic Hilbert spaces share most of the standard properties of complex Hilbert spaces. In particular, the Cauchy-Schwartz inequality holds on quaternionic Hilbert spaces as well as the Riesz representation theorem for their duals.  
Let $\D(A)$ denote the domain of $A$. $A$ is said to be right linear if
$$A(\phi q+\psi p)=(A\phi )q+(A\psi )p;\quad\forall\phi ,\psi \in\D(A), q, p\in \mathbb{H}.$$
The set of all right linear operators will be denoted by $\mathcal{L}(V_\mathbb{H}^R)$. 
We call an operator $A\in \mathcal{L}(V_\mathbb{H}^R)$ bounded if
\begin{equation*}
\|A\|=\sup_{\|\phi \|=1}\|A\phi \|<\infty.
\end{equation*}
or equivalently, there exists $K\geq 0$ such that $\|A\phi \|\leq K\|\phi \|$ for $\phi \in\D(A)$. The set of all bounded right linear operators will be denoted by $\B(V_\mathbb{H}^R)$.
\begin{proposition}\cite{KU}
Let $A\in \B(V_\mathbb{H}^R)$ and suppose that $\left\|A\right\|< 1.$ Then $(I_{V_\mathbb{H}^R}-A)^{-1}$ exists.
\end{proposition}

\begin{definition}\cite{KH}\label{D3}(\textit{Discrete Frames})	A countable family of elements $\left\{f_{k}\right\}_{k=1}^{m}$ in $V^{R}_{\mathbb{H}}$ is a frame for $V^{R}_{\mathbb{H}}$ if there exist constants $A,B>0$ such that
	\begin{equation}
	A\left\|f\right\|^{2}\leq\displaystyle\sum_{k=1}^{m}\left|\left\langle f|f_{k}\right\rangle\right|^{2}\leq B\left\|f\right\|^{2},
	\end{equation}
	$~\text{for all}~ f\in V^{R}_{\mathbb{H}}.$
\end{definition} 
\noindent Let $\left\{f_{k}\right\}^{m}_{k=1}$ be a frame in $V^{R}_{\mathbb{H}}$  and define a linear mapping $T:\mathbb{H}^{m}\longrightarrow V^{R}_{\mathbb{H}},~$ by
\begin{equation}\label{eq15}
T\left\{c_{k}\right\}^{m}_{k=1}=\displaystyle\sum_{k=1}^{m} f_{k}c_{k},~c_{k}\in \mathbb{H}.
\end{equation}
$T$ is usually called the \textit{pre-frame operator}, or the \textit{synthesis operator.}
The adjoint operator $T^{\dagger}:V^{R}_{\mathbb{H}}\longrightarrow \mathbb{H}^{m},$ given by
\begin{equation}\label{eq16}
T^{\dagger}f=\left\{\left\langle f|f_{k}\right\rangle\right\}^{m}_{k=1}	
\end{equation}
is called the \textit{analysis operator.}
By composing $T$ with its adjoint we obtain the \textit{frame operator} $S:V^{R}_{\mathbb{H}}\longrightarrow V^{R}_{\mathbb{H}},$ by
\begin{equation}\label{eq17}
~Sf=TT^{\dagger}f=\displaystyle\sum_{k=1}^{m} f_{k}\left\langle f|f_{k}\right\rangle	.
\end{equation}
\begin{theorem}\label{T2}\cite{Kho}
For each $q\in \mathbb{H},$ let the set $\{\eta_{q}^{i}: i=1,2,\cdots,n\}$ be  linearly independent in $V_{\mathbb{H}}^{R}.$ We define an operator $A$ by 
\begin{equation}\label{E3}
\sum_{i=1}^{n}\int_{\mathbb{H}}\left|\eta_{q}^{i}\right\rangle\left\langle \eta_{q}^{i}\right|d\mu(q)=A
\end{equation}
and we always assume that $A\in GL( V_{\mathbb{H}}^{R})$, where
 $$GL(V_{\mathbb{H}}^{R})=\left\{A:V_{\mathbb{H}}^{R}\longrightarrow V_{\mathbb{H}}^{R}: A\mbox{~bounded and~} A^{-1}\mbox{~bounded}\right\}.$$ Then the operator $A$ is positive and self adjoint.
\end{theorem}
\begin{definition}\cite{Kho}\textit{(Continuous frame)}\label{CFD1}
	A set of vectors $\{\eta^i_q\in V_\mathbb{H}^R~|~i=1,2,\cdots, n,~q\in \mathbb{H}\}$ constitute a rank $n$ right quaternionic continuous frame, denoted by $F(\eta^i_q, A, n)$, if
	\begin{enumerate}
		\item[(i)] for each $q\in \mathbb{H}$, the set of vectors $\{\eta^i_q\in V_\mathbb{H}^R~|~i=1,2,\cdots, n\}$ is a linearly independent set.
		\item[(ii)]there exists a positive operator $A\in GL(V_\mathbb{H}^R)$ such that
		$$\sum_{i=1}^{n}\int_{\mathbb{H}}|\eta^i_q\rangle\langle \eta^i_q|d\mu(q)=A.$$
	\end{enumerate}
\end{definition}
\begin{theorem}\cite{Kho}\label{T4}
	For $\phi\in V_\mathbb{H}^R$, we have
	\begin{equation}\label{E23}
	m(A)\|\phi\|^2\leq\sum_{i=1}^{n}\int_{\mathbb{H}}|\langle\eta_{q}^i|\phi\rangle|^2 d\mu(q)\leq M(A)\|\phi\|^2,
	\end{equation}
where $M(A)=\displaystyle\sup_{\|\phi\|=1}\langle\phi|A\phi\rangle$ and
$m(A)=\displaystyle\inf_{\|\phi\|=1}\langle\phi|A\phi\rangle.$
\end{theorem}
The inequality (\ref{E23}) presents the frame condition for the set of vectors
$$\{\eta^i_q\in V_\mathbb{H}^R~|~i=1,2,\cdots, n,~q\in \mathbb{H}\}$$
with frame bounds $m(A)$ and $M(A)$. 
\begin{theorem}\cite{Kho}(Frame decomposition)\label{FD}
Let $\{\eta^i_q\in V_\mathbb{H}^R~|~i=1,2,\cdots, n,~q\in \mathbb{H}\}$ be a rank $n$ continuous frame with bounds $m(A)$ and $M(A)$. Then for any $\phi\in V_\mathbb{H}^R, $ we have
\begin{eqnarray*}
\phi&=&\sum_{i=1}^{n}\int_\quat \eta^i_q\left\langle \phi|A^{-1}\eta^i_q\right\rangle d\mu(q)\\&=&\sum_{i=1}^{n}\int_\quat A^{-1}\eta^i_q\left\langle \phi|\eta^i_q\right\rangle d\mu(q),.
\end{eqnarray*}
where $A$ is the frame operator of  the frame $\{\eta^i_q\in V_\mathbb{H}^R~|~i=1,2,\cdots, n,~q\in \mathbb{H}\}.$
\end{theorem}
\begin{theorem}\cite{Kho}
Let $\{\eta^i_q\in V_\mathbb{H}^R~|~i=1,2,\cdots, n,~q\in \mathbb{H}\}$ be a rank $n$ continuous frame with bounds $m(A)$ and $M(A)$.Then $\{A^{-1}\eta^i_q\in V_\mathbb{H}^R~|~i=1,2,\cdots, n,~q\in \mathbb{H}\}$ is a rank $n$ continuous frame with bounds $\displaystyle\frac{1}{M(A)}$ and $\displaystyle\frac{1}{m(A)}.$
\end{theorem}

\begin{definition}
We call a family $\{\xi^i_q\in V_\mathbb{H}^R~|~i=1,2,\cdots, n,~q\in \mathbb{H}\}$ of elements in $V_\mathbb{H}^R$  a rank $n$ continuous Bessel family if there exists
$D>0$ such that
\begin{equation}
\sum_{i=1}^{n}\int_{\mathbb{H}}\left| \left\langle \xi^i_q|\phi\right\rangle \right|^{2}d\mu(q)\leq D\left\| \phi\right\|^{2},  
\end{equation}
for all $\phi\in V_\mathbb{H}^R.$	
\end{definition}
A rank $n$ continuous Bessel family $\{\xi^i_q\in V_\mathbb{H}^R~|~i=1,2,\cdots, n,~q\in \mathbb{H}\}$ will be called a rank $n$ continuous frame if there exists $C>0$ such that
\begin{equation}
C\left\|\phi\right\|^{2}\leq\sum_{i=1}^{n}\int_{\mathbb{H}}\left| \left\langle \xi^i_q|\phi\right\rangle \right|^{2}d\mu(q),
\end{equation}
for all $\phi\in V_\mathbb{H}^R.$\\
The following result is an adaptation of the discrete case considered in \cite{SH}.
\begin{theorem}\label{SH}
	Let $\{\zeta^i_q\in V_\mathbb{H}^R~|~i=1,2,\cdots, n,~q\in \mathbb{H}\}$ be a rank $n$ continuous Bessel family of  $V_\mathbb{H}^R$ with bound $D.$  Then the mapping $T$ from $\mathbb{H}^{n}$ to $V_\mathbb{H}^R$ defined by
	\begin{equation}\label{444}
	T(\{c_{i}\}_{i=1}^{n}):=\sum_{i=1}^{n}\int_{\mathbb{H}}\zeta^i_q c_{i} d\mu(q)
	\end{equation} 
	is a right linear and bounded operator with $\left\| T\right\|\leq \sqrt{D}. $
\end{theorem}
\begin{proof}
	It is not difficult to see that $T$ is right linear. Now for $\phi\in V_{\quat}^{R},$
	\begin{eqnarray*}
		\left\| T\{c_{i}\}\right\| &=&\sup_{\left\| \phi\right\| =1}\left|\left\langle T\{c_{i}\}|\phi\right\rangle   \right|\\
		&=&\sup_{\left\| \phi\right\| =1} \left|\left\langle \sum_{i=1}^{n}\int_{\mathbb{H}}\zeta^i_q c_{i} d\mu(q)|\phi\right\rangle   \right|\\
		&=&\sup_{\left\| \phi\right\| =1}\left|\sum_{i=1}^{n}\int_{\mathbb{H}}\overline{c_{i}}\left\langle \zeta^i_q|\phi\right\rangle d\mu(q) \right| \\
		&\leq&\sup_{\left\| \phi\right\| =1}\sum_{i=1}^{n}\int_{\mathbb{H}}\left| \overline{c_{i}}\left\langle \zeta^i_q|\phi\right\rangle\right|d\mu(q)\\
		&\leq&\sup_{\left\| \phi\right\| =1}\left(\sum_{i=1}^{n}\int_{\mathbb{H}}\left|\left\langle \zeta^i_q|\phi\right\rangle  \right|^{2} d\mu(q) \right) ^{\frac{1}{2}}\left(\sum_{i=1}^{n}\left| c_{i}\right|^{2} \right)^{\frac{1}{2}}\\
		&\leq&\sup_{\left\| \phi\right\| =1}\left(D\left\| \phi\right\|^{2} \right)^{\frac{1}{2}}\left(\sum_{i=1}^{n}\left| c_{i}\right|^{2} \right)^{\frac{1}{2}}\\
		&=&\sqrt{D}\left(\sum_{i=1}^{n}\left| c_{i}\right|^{2} \right)^{\frac{1}{2}}.
	\end{eqnarray*}	
	Hence $\left\| T\right\| \leq \sqrt{D}.$
\end{proof}
\noindent
By composing the operator $T$ in (\ref{444}) with its adjoint operator $T^{\dagger}$ we get the frame operator.
\noindent
\begin{proposition}
Let $A=\displaystyle\sum_{i=1}^{n}\int_{\mathbb{H}}|\eta^i_q\rangle\langle \zeta^i_q|d\mu(q).$ Then $A^{\dagger}=\displaystyle\sum_{i=1}^{n}\int_{\mathbb{H}}|\zeta^i_q\rangle\langle \eta^i_q|d\mu(q).$
\end{proposition}
\begin{proof}
For $\psi, \phi\in V_{\mathbb{H}}^{R},~A\psi=\displaystyle\sum_{i=1}^{n}\int_{\mathbb{H}}|\eta^i_q\rangle\langle \zeta^i_q|\psi\rangle d\mu(q),$ we have  
\begin{equation}
    \langle \phi|A\psi\rangle=\displaystyle\sum_{i=1}^{n}\int_{\mathbb{H}}\langle\phi|\eta^i_q\rangle\langle \zeta^i_q|\psi\rangle d\mu(q).
\end{equation}
If we take $T=\displaystyle\sum_{i=1}^{n}\int_{\mathbb{H}}|\zeta^i_q\rangle\langle \eta^i_q|d\mu(q)$ then $T\phi=\displaystyle\sum_{i=1}^{n}\int_{\mathbb{H}}|\zeta^i_q\rangle\langle \eta^i_q|\phi \rangle d\mu(q).$ Hence,
$$\langle\psi|T\phi\rangle=\displaystyle\sum_{i=1}^{n}\int_{\mathbb{H}}\langle \psi|\zeta^i_q\rangle\langle \eta^i_q|\phi \rangle d\mu(q).$$ Now 
\begin{eqnarray*}
\langle T\phi|\psi\rangle &=&\overline{\langle \psi|T\phi\rangle }\\
&=&\sum_{i=1}^{n}\int_{\mathbb{H}}\overline{\langle \psi|\zeta^i_q\rangle\langle \eta^i_q|\phi \rangle} d\mu(q).\\
&=&\sum_{i=1}^{n}\int_{\mathbb{H}}\overline{\langle \eta^i_q|\phi \rangle}~\overline{\langle \psi|\zeta^i_q\rangle} d\mu(q).\\
&=& \sum_{i=1}^{n}\int_{\mathbb{H}}\langle \phi| \eta^i_q\rangle\langle \zeta^i_q|\psi\rangle d\mu(q)\\
&=&\langle \phi|A\psi\rangle.
\end{eqnarray*}
Therefore, $\langle \phi|A\psi\rangle=\langle T\phi|\psi\rangle$ for all $\phi,\psi\in V_\quat^R$.  That is, $T=A^{\dagger}.$
\end{proof}
\section{Frame perturbation}
In this section we present perturbations of rank $n$ continuous frames in $V_\quat^R$ following the  frame perturbation theory presented, for complex discrete frames, in \cite{Chen, Ole3}.
\begin{theorem}\label{T7}
Let $\{\eta^i_q\in V_\mathbb{H}^R~|~i=1,2,\cdots, n,~q\in \mathbb{H}\}$ be a rank $n$ right quaternionic continuous frame with bounds $m(A)$ and $M(A)$ and frame operator $A$. Then any family $\{\zeta^i_q\in V_\mathbb{H}^R~|~i=1,2,\cdots, n,~q\in \mathbb{H}\}$ satisfying
\begin{equation}\label{Assum}
\kappa:=\sum_{i=1}^{n}\int_{\mathbb{H}}\left\|\eta^i_q-\zeta^i_q \right\|^2d\mu(q)<m(A) 
\end{equation} 
is a rank $n$ continuous frame for $V_\mathbb{H}^R$ with bounds $m(A)\left(1-\sqrt{\frac{\kappa}{m(A)}} \right)^{2}$  and  $M(A)\left( 1+\sqrt{\displaystyle\frac{\kappa}{M(A)}}\right)^{2}$.
\end{theorem}
\begin{proof}
Suppose that $\{\eta^i_q\in V_\mathbb{H}^R~|~i=1,2,\cdots, n,~q\in \mathbb{H}\}$ is a rank $n$ right quaternionic continuous frame with bounds $m(A)$ and $M(A).$ Then 
\begin{equation*}
m(A)\|\phi\|^2\leq\sum_{i=1}^{n}\int_{\mathbb{H}}|\langle\eta_{q}^i|\phi\rangle|^2 d\mu(q)\leq M(A)\|\phi\|^2.
\end{equation*}	
From \ref{Assum}, $\{\zeta^i_q\in V_\mathbb{H}^R~|~i=1,2,\cdots, n,~q\in \mathbb{H}\}$ is a continuous Bessel family in $V_\mathbb{H}^R.$ Thus, we can define an  operator $\mathfrak{U}:V_\mathbb{H}^R\longrightarrow V_\mathbb{H}^R$ by\vspace{-0.5cm}
\begin{equation}
\mathfrak{U}\phi=\sum_{i=1}^{n}\int_{\mathbb{H}}\zeta^i_q\left\langle \phi|A^{-1}\eta^i_q\right\rangle d\mu(q). 
\end{equation}
The operator $\mathfrak{U}$ is bounded.
To see it, let $\phi\in V_\mathbb{H}^R,$
\begin{eqnarray*}
\left\|\mathfrak{U}\phi \right\|^{2} &=&\left\langle \mathfrak{U}\phi|\mathfrak{U}\phi\right\rangle\\
&=&\left\langle \sum_{i=1}^{n}\int_{\mathbb{H}}\zeta^i_q\left\langle \phi|A^{-1}\eta^i_q\right\rangle d\mu(q)|\sum_{j=1}^{n}\int_{\mathbb{H}}\zeta^j_q\left\langle \phi \vert A^{-1}\eta^j_q\right\rangle d\mu(q)\right\rangle\\
&=&\sum_{i=1}^{n}\sum_{j=1}^{n}\int_{\mathbb{H}}\overline{\left\langle \phi|A^{-1}\eta^i_q\right\rangle}\left\langle  \zeta^i_q|\zeta^j_q\right\rangle \left\langle \phi| A^{-1}\eta^j_q\right\rangle d\mu(q)\\
&\leq& \left|\sum_{i=1}^{n}\sum_{j=1}^{n}\int_{\mathbb{H}}\overline{\left\langle \phi|A^{-1}\eta^i_q\right\rangle}\left\langle  \zeta^i_q|\zeta^j_q\right\rangle \left\langle \phi| A^{-1}\eta^j_q\right\rangle d\mu(q)\right|\\
&\leq&\sum_{i=1}^{n}\sum_{j=1}^{n}\int_{\mathbb{H}}\left|\overline{\left\langle \phi|A^{-1}\eta^i_q\right\rangle}\right|\left|\left\langle  \zeta^i_q|\zeta^j_q\right\rangle\right|\left|\left\langle \phi| A^{-1}\eta^j_q\right\rangle \right|d\mu(q)\\
&\leq&\alpha\sum_{i=1}^{n}\sum_{j=1}^{n}\int_{\mathbb{H}}\left|\overline{\left\langle \phi|A^{-1}\eta^i_q\right\rangle}\right|\left|\left\langle \phi| A^{-1}\eta^j_q\right\rangle \right|d\mu(q),\mbox{~where ~} \alpha=\max_{i,j}\sup_{q\in\quat}\left|\left\langle  \zeta^i_q|\zeta^j_q\right\rangle\right|\\
&\leq&\frac{\alpha}{2}\sum_{i=1}^{n}\sum_{j=1}^{n}\int_{\mathbb{H}}\left(\left|\overline{\left\langle \phi|A^{-1}\eta^i_q\right\rangle}\right|^{2}+\left|\left\langle \phi| A^{-1}\eta^j_q\right\rangle \right|^{2}\right)d\mu(q)\\
&\leq&\frac{n\alpha}{2}\sum_{i=1}^{n}\int_{\mathbb{H}}\left|\left\langle \phi|A^{-1}\eta^i_q\right\rangle\right|^{2}d\mu(q)+\frac{n\alpha}{2}\sum_{j=1}^{n}\int_{\mathbb{H}}\left|\left\langle \phi|A^{-1}\eta^j_q\right\rangle\right|^{2}d\mu(q)\\
&\leq&\frac{n\alpha}{2}\frac{1}{m(A)}\left\|\phi\right\|^{2}+\frac{n\alpha}{2}\frac{1}{m(A)}\left\|\phi\right\|^{2}\\
&=&\frac{n\alpha}{m(A)}\left\|\phi\right\|^{2}.
\end{eqnarray*}
It follows that there exists $K>0$ such that $\left\|\mathfrak{U}\phi \right\|\leq K\left\|\phi\right\|,$ where $K=\displaystyle\sqrt{\frac{n\alpha}{m(A)}}.$\\
Hence $\mathfrak{U}$ is bounded. Now, from theorem \ref{FD},
\begin{eqnarray*}
\left\|\phi-\mathfrak{U}\phi \right\|^{2} &=& \left\|\sum_{i=1}^{n}\int_{\mathbb{H}}\eta^i_q\left\langle \phi|A^{-1}\eta^i_q\right\rangle d\mu(q)-\sum_{i=1}^{n}\int_{\mathbb{H}}\zeta^i_q\left\langle \phi|A^{-1}\eta^i_q\right\rangle d\mu(q) \right\|^{2}\\
&=&\left\| \sum_{i=1}^{n}\int_{\mathbb{H}}(\eta^i_q-\zeta^i_q) \left\langle \phi|A^{-1}\eta^i_q\right\rangle d\mu(q)  \right\|^{2}\\
&\leq&\sum_{i=1}^{n}\int_{\mathbb{H}}\left|\left\langle \phi|A^{-1}\eta^i_q\right\rangle\right|^{2}\left\| \eta^i_q-\zeta^i_q    \right\|^{2}d\mu(q)\\
&\leq&\left( \sum_{i=1}^{n}\int_{\mathbb{H}}\left\| \eta^i_q-\zeta^i_q\right\| ^{2}d\mu(q)\right) \left( \sum_{i=1}^{n}\int_{\mathbb{H}}\left|\left\langle \phi|A^{-1}\eta^i_q\right\rangle\right |^{2} d\mu(q)\right) \\
&\leq&\kappa\frac{1}{m(A)}\left\| \phi\right\|^{2}. 
\end{eqnarray*}
That is, 
\begin{eqnarray*}
\left\|\phi-\mathfrak{U}\phi \right\|^{2} &=& \left\|\mathcal{I}_{V_\mathbb{H}^R}\phi-\mathfrak{U}\phi \right\|^{2}
=	\left\|(\mathcal{I}_{V_\mathbb{H}^R}-\mathfrak{U})\phi \right\|^{2}
\leq\frac{\kappa}{m(A)}\left\| \phi\right\|^{2}. 
\end{eqnarray*}
Therefore, $\displaystyle\left\|(\mathcal{I}_{V_\mathbb{H}^R}-\mathfrak{U})\phi \right\|\leq \sqrt{\frac{\kappa}{m(A)}}\left\| \phi\right\|.$ It follows that $\displaystyle\left\|\mathcal{I}_{V_\mathbb{H}^R}-\mathfrak{U} \right\|\leq \sqrt{\frac{\kappa}{m(A)}}.$ 
Thus $\displaystyle\left\|\mathcal{I}_{V_\mathbb{H}^R}-\mathfrak{U} \right\|\leq \sqrt{\frac{\kappa}{m(A)}}<1$ and $\mathfrak{U}$ is invertible. Also we have
$$\left|\left\| \mathfrak{U}\right\| - \left\|\mathcal{I}_{V_\mathbb{H}^R}\right\|  \right|\leq \left\|\mathcal{I}_{V_\mathbb{H}^R}-\mathfrak{U} \right\|\leq \sqrt{\frac{\kappa}{m(A)}}.  $$
Hence, $\displaystyle\left\| \mathfrak{U}\right\|\leq 1+\sqrt{\frac{\kappa}{m(A)}}  $ and $\displaystyle\left\| \mathfrak{U}^{-1}\right\|\leq \displaystyle\frac{1}{1-\sqrt{\frac{\kappa}{m(A)}}}.$
For $\phi\in V_\mathbb{H}^R,$
\begin{equation*}
\phi=\mathfrak{U}\mathfrak{U}^{-1}\phi=\sum_{i=1}^{n}\int_{\mathbb{H}}\zeta^i_q\left\langle \mathfrak{U}^{-1}\phi|A^{-1}\eta^i_q\right\rangle d\mu(q). 
\end{equation*}
Therefore
\begin{eqnarray*}
\left\| \phi\right\|^{4}&=&\left|\left\langle \sum_{i=1}^{n}\int_{\mathbb{H}}\zeta^i_q\left\langle \mathfrak{U}^{-1}\phi|A^{-1}\eta^i_q\right\rangle d\mu(q)|\phi\right\rangle  \right|^{2}\\
&=& \left|\sum_{i=1}^{n}\int_{\mathbb{H}}\overline{\left\langle \mathfrak{U}^{-1}\phi|A^{-1}\eta^i_q\right\rangle}\left\langle \zeta^i_q|\phi \right\rangle d\mu(q) \right|^{2}\\
&\leq&\left( \sum_{i=1}^{n}\int_{\mathbb{H}}\left|\overline{\left\langle \mathfrak{U}^{-1} \phi |A^{-1} \eta^i_q\right\rangle}   \right|^{2}d\mu(q)\right) \left( \sum_{i=1}^{n}\int_{\mathbb{H}}\left|\left\langle \zeta^i_q|\phi \right\rangle\right|^{2}d\mu(q)\right) \\
&\leq&\frac{1}{m(A)}\left\| \mathfrak{U}^{-1}\phi\right\|^{2} \cdot\sum_{i=1}^{n}\int_{\mathbb{H}}\left|\left\langle \zeta^i_q|\phi \right\rangle\right|^{2}d\mu(q)\\
&\leq&\frac{1}{m(A)}\left\| \mathfrak{U}^{-1}\right\|^{2}\left\| \phi\right\| ^{2} \cdot\sum_{i=1}^{n}\int_{\mathbb{H}}\left|\left\langle \zeta^i_q|\phi \right\rangle\right|^{2}d\mu(q)\\
&\leq&\frac{\left\| \phi\right\| ^{2}}{m(A)\left(1-\sqrt{\frac{\kappa}{m(A)}} \right)^{2} }\cdot\sum_{i=1}^{n}\int_{\mathbb{H}}\left|\left\langle \zeta^i_q|\phi \right\rangle\right|^{2}d\mu(q).
\end{eqnarray*}
Hence,
\begin{equation}\label{me1}
\sum_{i=1}^{n}\int_{\mathbb{H}}\left|\left\langle \zeta^i_q|\phi \right\rangle\right|^{2}d\mu(q)\geq m(A)\left(1-\sqrt{\frac{\kappa}{m(A)}} \right)^{2}\left\| \phi\right\| ^{2},
\end{equation}
for all $\phi\in V_\mathbb{H}^R.$
On the other hand define a right linear operator $T:\quat^n\longrightarrow V_\mathbb{H}^R$ by
\begin{equation}
T\{c_{i}\}_{i=1}^n:=\sum_{i=1}^{n}\int_{\mathbb{H}}\zeta^i_q c_{i} d\mu(q)
\end{equation}
The frame operator for $\{\zeta^i_q\in V_\mathbb{H}^R~|~i=1,2,\cdots, n,~q\in \mathbb{H}\}$ is $TT^{\dagger},$ so its optimal upper frame bound  is $\left\|T \right\|^{2}. $
For $\{c_{i}\}\in \quat^n,$
\begin{eqnarray*}
\left\| T\{c_{i}\}\right\| &=&\left\| \sum_{i=1}^{n}\int_{\mathbb{H}}\zeta^i_q c_{i} d\mu(q)\right\|\\
&\leq& \left\| \sum_{i=1}^{n}\int_{\mathbb{H}}(\zeta^i_q-\eta^i_q) c_{i} d\mu(q)\right\| +\left\| \sum_{i=1}^{n}\int_{\mathbb{H}}\eta^i_q c_{i} d\mu(q)\right\| \\
&\leq&(\sqrt{M(A)}+\sqrt{\kappa})\left\| \{c_{i}\}\right\|.
\end{eqnarray*}
Hence $\left\| T\{c_{i}\}\right\|\leq (\sqrt{M(A)}+\sqrt{\kappa})\left\| \{c_{i}\}\right\|$ and $\left\| T\right\| ^{2}\leq (\sqrt{M(A)}+\sqrt{\kappa})^{2}. $
Therefore $\left\| T\right\| ^{2}\leq M(A)\left( 1+\sqrt{\displaystyle\frac{\kappa}{M(A)}}\right)^{2} . $
Since $\left\|T \right\|^{2}$  is the optimal upper frame bound for $\{\zeta^i_q\in V_\mathbb{H}^R~|~i=1,2,\cdots, n,~q\in \mathbb{H}\},$ we have
\begin{equation}\label{me2}
\sum_{i=1}^{n}\int_{\mathbb{H}}\left|\left\langle \zeta^i_q|\phi \right\rangle\right|^{2}d\mu(q)\leq M(A)\left( 1+\sqrt{\displaystyle\frac{\kappa}{M(A)}}\right)^{2}\left\| \phi\right\| ^{2} ,
\end{equation}
for all $\phi\in V_\mathbb{H}^R.$
From \ref{me1} and \ref{me2}, we get
\begin{equation}
m(A)\left(1-\sqrt{\frac{\kappa}{m(A)}} \right)^{2}\left\| \phi\right\| ^{2}\leq\sum_{i=1}^{n}\int_{\mathbb{H}}\left|\left\langle \zeta^i_q|\phi \right\rangle\right|^{2}d\mu(q)\leq M(A)\left( 1+\sqrt{\displaystyle\frac{\kappa}{M(A)}}\right)^{2} \left\| \phi\right\| ^{2},
\end{equation} 
for all $\phi\in V_\mathbb{H}^R.$
Hence $\{\zeta^i_q\in V_\mathbb{H}^R~|~i=1,2,\cdots, n,~q\in \mathbb{H}\}$ is a rank $n$ continuous frame with bounds $m(A)\left(1-\sqrt{\frac{\kappa}{m(A)}} \right)^{2}$  and  $M(A)\left( 1+\sqrt{\displaystyle\frac{\kappa}{M(A)}}\right)^{2}.$
\end{proof}
\begin{theorem}
Let $\{\eta^i_q\in V_\mathbb{H}^R~|~i=1,2,\cdots, n,~q\in \mathbb{H}\}$ be a rank $n$ right quaternionic continuous frame with bounds $m(A)$ and $M(A).$ Let $\{\zeta^i_q\in V_\mathbb{H}^R~|~i=1,2,\cdots, n,~q\in \mathbb{H}\}$ be any family defined in (\ref{Assum}).Then $\{\eta^i_q+\zeta^i_q\in V_\mathbb{H}^R~|~i=1,2,\cdots, n,~q\in \mathbb{H}\}$ is a rank $n$ continuous frame in $V_\mathbb{H}^R.$
\end{theorem}
\begin{proof}
For $\phi\in V_{\mathbb{H}}^{R},$ we have
\begin{eqnarray*}
\sum_{i=1}^{n}\int_{\mathbb{H}}\left|\left\langle \eta^i_q+\zeta^i_q|\phi\right\rangle\right|^{2}d\mu(q)
&=&\sum_{i=1}^{n}\int_{\mathbb{H}}\left|\left\langle \phi|\eta^i_q+\zeta^i_q\right\rangle\right|^{2}d\mu(q)\\
&=&\sum_{i=1}^{n}\int_{\mathbb{H}}\left| \left\langle \phi|\eta^i_q\right\rangle+ \left\langle \phi|\zeta^i_q\right\rangle\right|^{2}d\mu(q)\\
&\leq&\sum_{i=1}^{n}\int_{\mathbb{H}}\left| \left\langle \eta^i_q|\phi\right\rangle\right|^{2}d\mu(q)+\sum_{i=1}^{n}\int_{\mathbb{H}}\left| \left\langle \zeta^i_q|\phi\right\rangle\right|^{2}d\mu(q).
\end{eqnarray*}
Since $\{\eta^i_q\in V_\mathbb{H}^R~|~i=1,2,\cdots, n,~q\in \mathbb{H}\}$ and $\{\zeta^i_q\in V_\mathbb{H}^R~|~i=1,2,\cdots, n,~q\in \mathbb{H}\}$  are rank $n$ continuous frames in $V_\mathbb{H}^R,$ for  $\phi\in V_\mathbb{H}^R,$
\begin{equation*}
 m(A)\|\phi\|^2\leq\sum_{i=1}^{n}\int_{\mathbb{H}}|\langle\eta_{q}^i|\phi\rangle|^2 d\mu(q)\leq M(A)\|\phi\|^2   
\end{equation*}
and
\begin{equation*}
m(A)\left(1-\sqrt{\frac{\kappa}{m(A)}} \right)^{2}\|\phi\|^2\leq\sum_{i=1}^{n}\int_{\mathbb{H}}\left|\left\langle \zeta^i_q|\phi \right\rangle\right|^{2}d\mu(q)\leq M(A)\left( 1+\sqrt{\displaystyle\frac{\kappa}{M(A)}}\right)^{2}\|\phi\|^2 .
\end{equation*} 
Therefore
\begin{equation*}
  \sum_{i=1}^{n}\int_{\mathbb{H}}\left|\left\langle \eta^i_q+\zeta^i_q|\phi\right\rangle\right|^{2}d\mu(q)\leq M(A)\left[1+\left( 1+\sqrt{\displaystyle\frac{\kappa}{M(A)}}\right)^{2}\right]\|\phi\|^{2}.  
\end{equation*}
Similarly one can obtain
\begin{equation*}
m(A)\left[1+\left( 1-\sqrt{\displaystyle\frac{\kappa}{m(A)}}\right)^{2}\right]\|\phi\|^{2}\leq  \sum_{i=1}^{n}\int_{\mathbb{H}}\left|\left\langle \eta^i_q+\zeta^i_q|\phi\right\rangle\right|^{2}d\mu(q).  
\end{equation*}
Hence 
\begin{equation*}
m(A)\left[1+\left( 1-\sqrt{\displaystyle\frac{\kappa}{m(A)}}\right)^{2}\right]\|\phi\|^{2}\leq  \sum_{i=1}^{n}\int_{\mathbb{H}}\left|\left\langle \eta^i_q+\zeta^i_q|\phi\right\rangle\right|^{2}d\mu(q)\leq M(A)\left[1+\left( 1+\sqrt{\displaystyle\frac{\kappa}{M(A)}}\right)^{2}\right]\|\phi\|^{2}, 
\end{equation*}
for all $\phi\in V_{\mathbb{H}}^{R}.$ Therefore,  $\{\eta^i_q+\zeta^i_q\in V_\mathbb{H}^R~|~i=1,2,\cdots, n,~q\in \mathbb{H}\}$ is a rank $n$ continuous frame in  $V_\mathbb{H}^R$ with bounds $m(A)\left[1+\left( 1-\sqrt{\displaystyle\frac{\kappa}{m(A)}}\right)^{2}\right]$ and $M(A)\left[1+\left( 1+\sqrt{\displaystyle\frac{\kappa}{M(A)}}\right)^{2}\right].$ 
\end{proof}
\begin{proposition}
Frame operator for $\{\eta^i_q+\zeta^i_q\in V_\mathbb{H}^R~|~i=1,2,\cdots, n,~q\in \mathbb{H}\}$ is 
\begin{equation}
 A^{\prime}= \sum_{i=1}^{n}\int_{\mathbb{H}}|\eta^i_q+\zeta^i_q\rangle\langle \eta^i_q+\zeta^i_q|d\mu(q). 
\end{equation}
Then $A^{\prime}$ is self adjoint and positive.
\end{proposition}
\begin{proof}
For $\phi\in V_{\mathbb{H}}^{R}, A^{\prime}\left | \phi\right\rangle= \sum_{i=1}^{n}\int_{\mathbb{H}}|\eta^i_q+\zeta^i_q\rangle\langle \eta^i_q+\zeta^i_q|\phi \rangle d\mu(q).$ 
Since $\langle A^{\prime}\phi|\psi\rangle=(A^{\prime}|\phi\rangle)^{\dagger}|\psi\rangle,$
\begin{eqnarray*}
\langle A^{\prime}\phi|\psi\rangle &=&\left(\sum_{i=1}^{n}\int_{\mathbb{H}}|\eta^i_q+\zeta^i_q\rangle\langle \eta^i_q+\zeta^i_q|\phi \rangle d\mu(q)\right)^{\dagger}(|\psi\rangle)\\
&=&\sum_{i=1}^{n}\int_{\mathbb{H}}\langle \phi| \eta^i_q+\zeta^i_q\rangle\langle \eta^i_q+\zeta^i_q| \psi\rangle d\mu(q)\\
&=&\langle \phi| A^{\prime}\psi\rangle.
\end{eqnarray*}
Hence $\langle A^{\prime}\phi|\psi\rangle=\langle \phi| A^{\prime}\psi\rangle$ for all $\phi,\psi\in V_\quat^R$. It follows that $A^{\prime}$ is self adjoint.
 Now, for each $\phi\in V_\quat^R$, 
 \begin{eqnarray*}
 \langle A^{\prime}\phi|\phi\rangle &=& \sum_{i=1}^{n}\int_{\mathbb{H}}\langle \phi| \eta^i_q+\zeta^i_q\rangle\langle \eta^i_q+\zeta^i_q| \phi\rangle d\mu(q)\\
 &=&\sum_{i=1}^{n}\int_{\mathbb{H}}\left |\langle \phi| \eta^i_q+\zeta^i_q\rangle\right |^{2}d\mu(q)\\
 &\geq& 0.
 \end{eqnarray*}
 Thereby $A^{\prime}$ is positive.
\end{proof}
The following results are the quaternionic continuous counterparts of certain perturbations  considered for complex discrete frames in \cite{Chen}.
\begin{theorem}\label{MT}
	Let $\{\eta^i_q\in V_\mathbb{H}^R~|~i=1,2,\cdots, n,~q\in \mathbb{H}\}$ be a rank $n$ right quaternionic continuous frame with bounds $m(A),M(A)$ and $\{\overline{\eta}^i_q\in V_\mathbb{H}^R~|~i=1,2,\cdots, n,~q\in \mathbb{H}\}$ be the dual frame of $\{\eta^i_q\in V_\mathbb{H}^R~|~i=1,2,\cdots, n,~q\in \mathbb{H}\}$ with bounds $C,D.$ Assume that the family $\{\Psi^i_q\in V_\mathbb{H}^R~|~i=1,2,\cdots, n,~q\in \mathbb{H}\}$ satisfies the following two conditions:
	\begin{enumerate}
		\item $\lambda:=\displaystyle\sum_{i=1}^{n}\int_{\mathbb{H}}\left\|\eta^i_q-\Psi^i_q \right\|^{2}d\mu(q)<\infty; $
		\item $\gamma:=\displaystyle\sum_{i=1}^{n}\int_{\mathbb{H}}\left\|\eta^i_q-\Psi^i_q \right\|\left\|\overline{\eta}^i_q\right\|d\mu(q)<1. $
	\end{enumerate}
Then $\{\Psi^i_q\in V_\mathbb{H}^R~|~i=1,2,\cdots, n,~q\in \mathbb{H}\}$ is a rank $n$ continuous frame for $V_\mathbb{H}^R$ with bounds $\displaystyle\frac{(1-\gamma)^{2}}{D}$ and $\displaystyle M(A)\left(1+\sqrt{\frac{\lambda}{M(A)}} \right)^{2}.$
\end{theorem}
\begin{proof}
Let $T:\mathbb{H}^{n}\longrightarrow V_{\mathbb{H}}^{R}$ defined by 
\begin{equation*}
T(\{c_{i}\}_{i=1}^{n})=\sum_{i=1}^{n}\int_{\mathbb{H}}\eta^i_q c_{i}~d\mu(q),\mbox{~where~} \{c_{i}\}_{i=1}^{n}\in\quat^n
\end{equation*}
 be the pre-frame operator of the frame $\{\eta^i_q\in V_\mathbb{H}^R~|~i=1,2,\cdots, n,~q\in \mathbb{H}\}.$
From theorem \ref{SH}, $\left\|T \right\| \leq \sqrt{M(A)}.$	
Now define $U:\mathbb{H}^{n}\longrightarrow V_{\mathbb{H}}^{R}$ by 
\begin{equation*}
U(\{c_{i}\}_{i=1}^{n})=\sum_{i=1}^{n}\int_{\mathbb{H}}\Psi^i_q c_{i}~d\mu(q),\mbox{~where~} \{c_{i}\}_{i=1}^{n}\in\quat^n.
\end{equation*}
We have 
\begin{eqnarray*}
\left\| U\right\| &=&\sup_{\|\phi\|=1}\left\| U\phi\right\| \\
&=&\sup_{\|\phi\|=1}\left\| \sum_{i=1}^{n}\int_{\mathbb{H}}\Psi^i_q c_{i}d\mu(q)\right\|,\mbox{~where~} \phi=\{c_{i}\}_{i=1}^{n}\in\quat^n\\
&=&\sup_{\|\phi\|=1}\left\| \sum_{i=1}^{n}\int_{\mathbb{H}}(\Psi^i_q-\eta^i_q+\eta^i_q) c_{i}d\mu(q)\right\|\\
&\leq&\sup_{\|\phi\|=1}\left\| \sum_{i=1}^{n}\int_{\mathbb{H}}(\Psi^i_q-\eta^i_q) c_{i}d\mu(q)\right\|+\sup_{\|\phi\|=1}\left\| \sum_{i=1}^{n}\int_{\mathbb{H}}\eta^i_q c_{i}d\mu(q)\right\|\\
&\leq&\sup_{\|\phi\|=1}\sum_{i=1}^n|c_i|\int_\quat \|\Psi^i_q-\eta^i_q\|~d\mu(q)+\|T\|\\
&\leq&\sqrt{\lambda}+\left\| T\right\| ,\mbox{~by~} (1)\\
&\leq& \sqrt{\lambda}+\sqrt{M(A)}.
\end{eqnarray*}
Hence $U$ is well defined and $\left\| U\right\| \leq \sqrt{\lambda}+\sqrt{M(A)}. $
Now the adjoint $U^{\dagger}$  of $U$ can be defined by 
\begin{equation*}
U^{\dagger} :V_{\mathbb{H}}^{R} \longrightarrow \mathbb{H}^{n} \mbox{~by~} U^{\dagger}(\phi)=\{\left\langle \phi|\Psi_{q}^{i}\right\rangle \}_{i=1}^{n},~\forall \phi\in V_{\mathbb{H}}^{R}.
\end{equation*}
We have 
\begin{eqnarray*}
\sum_{i=1}^{n}\int_{\mathbb{H}}\left|\left\langle \phi|\Psi_{q}^{i}\right\rangle \right|^{2}d\mu(q) &=& \left\| U^{\dagger}	\phi\right\| ^{2}\\
&\leq&\left\| U^\dagger\|^2~\|\phi\right\| ^{2}\\
&=&\{\left\| U\right\|\left\| \phi\right\|\}^{2}\\
&\leq&\{\sqrt{\lambda}+\sqrt{M(A)}\}^{2}\left\| \phi\right\|^{2}\\
&=&M(A)\left(1+\sqrt{\frac{\lambda}{M(A)}} \right)^{2} \left\| \phi\right\|^{2}.
\end{eqnarray*}	
Therefore
\begin{equation}\label{fr1}
\sum_{i=1}^{n}\int_{\mathbb{H}}\left|\left\langle \phi|\Psi_{q}^{i}\right\rangle \right|^{2}d\mu(q)\leq M(A)\left(1+\sqrt{\frac{\lambda}{M(A)}} \right)^{2} \left\| \phi\right\|^{2}.
\end{equation}
Now define $L:V_{\mathbb{H}}^{R}\longrightarrow V_{\mathbb{H}}^{R}$ by 
\begin{equation*}
L(\phi)=\sum_{i=1}^{n}\int_{\mathbb{H}}\Psi_{q}^{i}\left\langle \phi|\overline{\eta}_{q}^{i}\right\rangle d\mu(q),~\forall\phi\in V_{\mathbb{H}}^{R}.
\end{equation*}
For $\phi\in V_{\mathbb{H}}^{R},$
\begin{eqnarray*}
\left\| \phi-L(\phi)\right\| &=&\left\|\sum_{i=1}^{n}\int_{\mathbb{H}} \eta^i_q\left\langle \phi|\overline{ \eta}^i_q\right\rangle d\mu(q) -\sum_{i=1}^{n}\int_{\mathbb{H}} \Psi^i_q\left\langle \phi|\overline{ \eta}^i_q\right\rangle d\mu(q) \right\|\\
&=&	\left\|\sum_{i=1}^{n}\int_{\mathbb{H}}(\eta^i_q-\Psi^i_q)
\left\langle \phi|\overline{ \eta}^i_q\right\rangle d\mu(q) \right\|\\
&\leq&\sum_{i=1}^{n}\int_{\mathbb{H}}\left\|(\eta^i_q-\Psi^i_q)
\left\langle \phi|\overline{ \eta}^i_q\right\rangle  \right\|d\mu(q)\\
&=&\sum_{i=1}^{n}\int_{\mathbb{H}}\left\|\eta^i_q-\Psi^i_q\right\|\left|\left\langle \phi|\overline{ \eta}^i_q\right\rangle  \right| d\mu(q)\\
&\leq&\sum_{i=1}^{n}\int_{\mathbb{H}}\left\|\eta^i_q-\Psi^i_q\right\|\left\| \phi\right\| \left\| \overline{ \eta}^i_q\right\| d\mu(q)\\
&=&\gamma\left\| \phi\right\|. 
\end{eqnarray*}
That is $\left\| \phi-L(\phi)\right\|\leq \gamma\left\| \phi\right\|,$ for all $\phi\in V_{\mathbb{H}}^{R}.$ It follows that $\left\| I_{V_{\mathbb{H}}^{R}}-L\right\| \leq \gamma$ and $\left\| I_{V_{\mathbb{H}}^{R}}-L\right\| \leq 1.$
So that $\left\| L\right\| \leq 1+\gamma$ and $\left\| L^{-1}\right\|\leq\displaystyle\frac{1}{1-\gamma}. $
Each $\phi\in V_{\mathbb{H}}^{R}$ can be written as 
\begin{eqnarray*}
\phi&=&LL^{-1}\phi\\
	&=&\sum_{i=1}^{n}\int_{\mathbb{H}}\Psi^i_q \left\langle L^{-1}\phi|\overline{ \eta}^i_q\right\rangle d\mu(q). 
\end{eqnarray*}
Now
\begin{eqnarray*}
\left\|\phi\right\|^{2} &=&\left\langle \phi|\phi\right\rangle \\
&=&\left\langle \phi\vert\sum_{i=1}^{n}\int_{\mathbb{H}}\Psi^i_q \left\langle L^{-1}\phi|\overline{ \eta}^i_q\right\rangle d\mu(q) \right\rangle \\
&=&\sum_{i=1}^{n}\int_{\mathbb{H}}\left\langle \phi|\Psi^i_q\right\rangle\left\langle L^{-1}\phi|\overline{ \eta}^i_q\right\rangle d\mu(q)\\
&\leq& \left( \sum_{i=1}^{n}\int_{\mathbb{H}}\left| \left\langle \phi|\Psi_{q}^{i}\right\rangle \right|^{2} d\mu(q)\right)^{\frac{1}{2}} \left( \sum_{i=1}^{n}\int_{\mathbb{H}}\left| \left\langle L^{-1}\phi|\overline{ \eta}_{q}^{i}\right\rangle \right|^{2} d\mu(q)\right)^{\frac{1}{2}}\\
&\leq&  \left( \sum_{i=1}^{n}\int_{\mathbb{H}}\left| \left\langle \phi|\Psi_{q}^{i}\right\rangle \right|^{2} d\mu(q)\right)^{\frac{1}{2}} \left(D. \left\| L^{-1}\phi\right\| ^{2}\right)^{\frac{1}{2}}\\
&=& \sqrt{D}\left\| L^{-1}\phi\right\| \left( \sum_{i=1}^{n}\int_{\mathbb{H}}\left| \left\langle \phi|\Psi_{q}^{i}\right\rangle \right|^{2} d\mu(q)\right)^{\frac{1}{2}}\\
&\leq&\sqrt{D}\frac{1}{1-\gamma}\left\| \phi\right\| \left( \sum_{i=1}^{n}\int_{\mathbb{H}}\left| \left\langle \phi|\Psi_{q}^{i}\right\rangle \right|^{2} d\mu(q)\right)^{\frac{1}{2}}.
\end{eqnarray*}
It follows that 
\begin{equation}\label{fr2}
\sum_{i=1}^{n}\int_{\mathbb{H}}\left| \left\langle \phi|\Psi_{q}^{i}\right\rangle \right|^{2} d\mu(q)\geq\frac{(1-\gamma)^{2}}{D}\left\| \phi\right\| ^{2}.
\end{equation}
From \ref{fr1} and \ref{fr2}, we get
\begin{equation}
\frac{(1-\gamma)^{2}}{D}\left\| \phi\right\| ^{2}\leq\sum_{i=1}^{n}\int_{\mathbb{H}}\left| \left\langle \phi|\Psi_{q}^{i}\right\rangle \right|^{2} d\mu(q)\leq M(A)\left(1+\sqrt{\frac{\lambda}{M(A)}} \right)^{2} \left\| \phi\right\|^{2},
\end{equation}
for all $\phi\in V_{\mathbb{H}}^{R}.$
Hence $\{\Psi^i_q\in V_\mathbb{H}^R~|~i=1,2,\cdots, n,~q\in \mathbb{H}\}$ is a rank $n$ continuous frame for $V_\mathbb{H}^R$ with bounds $\displaystyle\frac{(1-\gamma)^{2}}{D}$ and $\displaystyle M(A)\left(1+\sqrt{\frac{\lambda}{M(A)}} \right)^{2}.$
\end{proof}
\begin{definition}
Let $K$ and $L$ be subspaces of $V_{\quat}^{R}.$ When $K\neq\{0\},$ the gap from $K$ to $L$ is given by \begin{equation*}
\delta(K,L):=\sup_{\phi\in K,\|\phi \|=1}\inf_{\psi\in L}\left\| \phi-\psi\right\|. 
\end{equation*}
Also when $K=\{0\},$ we define $\delta(K,L)=0.$ 
\end{definition}
\begin{theorem}
Let $\{\eta^i_q\in V_\mathbb{H}^R~|~i=1,2,\cdots, n,~q\in \mathbb{H}\}$ be a rank $n$ continuous frame in $V_{\quat}^{R}$ with bounds $m(A)$ and $M(A)$ and let 	$\{\overline{\eta}^i_q\in V_\mathbb{H}^R~|~i=1,2,\cdots, n,~q\in \mathbb{H}\}$ be the dual frame of $\{\eta^i_q\in V_\mathbb{H}^R~|~i=1,2,\cdots, n,~q\in \mathbb{H}\}$ with bounds $C,D.$ Suppose that $\{\Psi^i_q\in V_\mathbb{H}^R~|~i=1,2,\cdots, n,~q\in \mathbb{H}\}$ is a family in $V_{\quat}^{R}.$ Let $K=\overline{\mbox{rightspan}}\{\Psi^i_q\}_{i=1}^{n}, L=\overline{\mbox{rightspan}}\{\eta^i_q\}_{i=1}^{n}$, where $q\in\quat$ and the right span is taken over $\quat$.  Assume that $\delta(K,L)<1.$ If $\{\Psi^i_q\in V_\mathbb{H}^R~|~i=1,2,\cdots, n,~q\in \mathbb{H}\}$ satisfies the following conditions:
\begin{enumerate}
	\item $\lambda:=\displaystyle\sum_{i=1}^{n}\int_{\mathbb{H}}\left\|\eta^i_q-\Psi^i_q \right\|^{2}d\mu(q)<\infty; $
	\item $\gamma:=\displaystyle\sum_{i=1}^{n}\int_{\mathbb{H}}\left\|\eta^i_q-\Psi^i_q \right\|\left\|\overline{\eta}^i_q\right\|d\mu(q)<1. $
\end{enumerate}
Then $\{\Psi^i_q\in V_\mathbb{H}^R~|~i=1,2,\cdots, n,~q\in \mathbb{H}\}$ is a rank $n$ continuous frame  with bounds $\displaystyle\frac{(1-\gamma)^{2}}{D}$ and $\displaystyle M(A)\left(1+\sqrt{\frac{\lambda}{M(A)}} \right)^{2}\frac{1}{(1-\delta(K,L))^{2}}.$ Moreover, the restriction of the orthogonal projection $P_{L}$ to $K$ is an isomorphism from $K$ onto $L.$   
\end{theorem}
\begin{proof}
Let $h\in K$ then $h=h_{L}+h-h_{L},$ where $h_{L}\in L$ with $P_{L}h=h_{L},$ $h=P_{L}h+h-h_{L}.$ Therefore
\begin{eqnarray*}
\left\|P_{L}h \right\|&\geq&\left\| h\right\|-\left\|h-h_{L} \right\|\\
&=&\left\| h\right\|-\left\| h\right\|\left\| \frac{h}{\left\| h\right\|}-\frac{h_{L}}{\left\| h\right\|}\right\| \\
&\geq & \left\| h\right\|-\left\| h\right\|\sup_{\phi\in K, \|\phi\|=1}\inf_{\psi\in L}\|\phi-\psi\|\\
&=&\left\| h\right\| -\left\| h\right\| \delta(K,L)\\
&=&(1-\delta(K,L))\left\| h\right\|. 
\end{eqnarray*}
Hence $\left\|P_{L}h \right\|\geq(1-\delta(K,L))\left\| h\right\|,$ for all $h\in K.$ Let $P_{L}(\Psi^i_q)=\Psi^i_{qL}$ then $\eta^i_q-\Psi^i_q=\eta^i_q-\Psi^i_{qL}+(\Psi^i_{qL}-\Psi^i_q).$ It follows that 
\begin{eqnarray*}
\left\| \eta^i_q-\Psi^i_q\right\| &\geq&\left\| \eta^i_q-\Psi^i_{qL}\right\|-\left\|\Psi^i_q-\Psi^i_{qL} \right\|\\
&\geq&\left\| \eta^i_q-\Psi^i_{qL}\right\|-\left\|\Psi^i_q \right\|\delta(K,L) \\
&\geq& \left\| \eta^i_q-\Psi^i_{qL}\right\|-(\left\|P_{L}(\Psi^i_q) \right\| -\left\| \Psi^i_q\right\| ) \\
&\geq&\left\| \eta^i_q-\Psi^i_{qL}\right\|\mbox{~as~} \left\|P_{L}(\Psi^i_q) \right\| -\left\| \Psi^i_q\right\|\leq 0.
\end{eqnarray*}
Therefore $\left\| \eta^i_q-P_{L}(\Psi^i_q)\right\|\leq \left\| \eta^i_q-\Psi^i_q\right\|.$ Hence,
 $$\displaystyle\sum_{i=1}^{n}\int_{\mathbb{H}}\left\| \eta^i_q-P_{L}(\Psi^i_q)\right\|^{2}d\mu(q)\leq \displaystyle\sum_{i=1}^{n}\int_{\mathbb{H}}\left\| \eta^i_q-\Psi^i_q\right\|^{2}d\mu(q) $$ and $$\displaystyle\sum_{i=1}^{n}\int_{\mathbb{H}}\left\| \eta^i_q-P_{L}(\Psi^i_q)\right\|\left\|\overline{\eta}^i_q \right\|d\mu(q)\leq \displaystyle\sum_{i=1}^{n}\int_{\mathbb{H}}\left\| \eta^i_q-\Psi^i_q\right\|\left\|\overline{\eta}^i_q  \right\|d\mu(q).$$ We apply theorem \ref{MT} to the sequence $\{P_{L}(\Psi^i_q)\}_{i=1}^{n}$ in $L$ and to the frame $\{\eta^i_q\}_{i=1}^{n}$ for $L$ to obtain $\{P_{L}(\Psi^i_q)\}_{i=1}^{n}$ as a frame for $L$ with bounds $\displaystyle\frac{(1-\gamma)^{2}}{D}$ and $\displaystyle M(A)\left(1+\sqrt{\frac{\lambda}{M(A)}} \right)^{2}.$ We have $P_{L}(K)=L$ and hence the restriction $Q:=P_{L}|_{K}$ is an isomorphism from $K$ onto $L.$  Now the claim is $\{\Psi^i_q\in V_\mathbb{H}^R~|~i=1,2,\cdots, n,~q\in \mathbb{H}\}$ is a frame for $K.$ For $\Psi\in K,$ we have
\begin{eqnarray*}
\sum_{i=1}^{n}\int_{\mathbb{H}}\left|\left\langle \Psi|\Psi^i_{q}\right\rangle  \right|^{2}d\mu(q) &=&\sum_{i=1}^{n}\int_{\mathbb{H}}\left|\left\langle \Psi|Q^{-1}Q(\Psi^i_{q})\right\rangle  \right|^{2}d\mu(q)\\
&=&\sum_{i=1}^{n}\int_{\mathbb{H}}\left|\left\langle(Q^{-1})^{\dagger} \Psi|Q(\Psi^i_{q})\right\rangle  \right|^{2}d\mu(q)\\
&\leq&M(A)\left(1+\sqrt{\frac{\lambda}{M(A)}} \right)^{2}\left\|(Q^{-1})^{\dagger} \Psi \right\|^{2} \\
&\leq&M(A)\left(1+\sqrt{\frac{\lambda}{M(A)}} \right)^{2}\left\lbrace \left\| (Q^{-1})^{\dagger}\right\|\left\|\Psi \right\| \right\rbrace ^{2}\\
&=&M(A)\left(1+\sqrt{\frac{\lambda}{M(A)}} \right)^{2}\left\| Q^{-1}\right\|^{2} \left\| \Psi\right\|^{2} \\
&\leq&M(A)\left(1+\sqrt{\frac{\lambda}{M(A)}} \right)^{2}\frac{1}{(1-\delta(K,L))^{2}}\left\| \Psi\right\|^{2} 
\end{eqnarray*}
Hence 
\begin{equation}
\sum_{i=1}^{n}\int_{\mathbb{H}}\left|\left\langle \Psi|\Psi^i_{q}\right\rangle  \right|^{2}d\mu(q)\leq M(A)\left(1+\sqrt{\frac{\lambda}{M(A)}} \right)^{2}\frac{1}{(1-\delta(K,L))^{2}}\left\| \Psi\right\|^{2} ,
\end{equation}
for all $\Psi\in K.$
Now 
\begin{eqnarray*}
\sum_{i=1}^{n}\int_{\mathbb{H}}\left|\left\langle \Psi|\Psi^i_{q}\right\rangle  \right|^{2}d\mu(q) &=&\sum_{i=1}^{n}\int_{\mathbb{H}}\left|\left\langle (Q^{-1})^\dagger\Psi|Q(\Psi^i_{q})\right\rangle  \right|^{2}d\mu(q)\\
&\geq&\frac{(1-\gamma)^{2}}{D}\left\| (Q^{-1})^\dagger\Psi\right\| ^{2}\\
&=&\frac{(1-\gamma)^{2}}{D}\left\| (Q^\dagger)^{-1}\Psi\right\| ^{2}\\
&\geq&\frac{(1-\gamma)^{2}}{D}\left\| \Psi\right\| ^{2}.
\end{eqnarray*}	
Hence 
\begin{equation}
\sum_{i=1}^{n}\int_{\mathbb{H}}\left|\left\langle \Psi|\Psi^i_{q}\right\rangle  \right|^{2}d\mu(q)\geq\frac{(1-\gamma)^{2}}{D}\left\| \Psi\right\| ^{2},
\end{equation}
for all $\Psi\in K.$ Therefore, $\{\Psi^i_q\in V_\mathbb{H}^R~|~i=1,2,\cdots, n,~q\in \mathbb{H}\}$ is a rank $n$ continuous frame for $K$ with bounds $\displaystyle\frac{(1-\gamma)^{2}}{D}$ and $\displaystyle M(A)\left(1+\sqrt{\frac{\lambda}{M(A)}} \right)^{2}\frac{1}{(1-\delta(K,L))^{2}}.$
\end{proof}
\begin{definition}
	We call a sequence $\{\eta^i_q\in V_\mathbb{H}^R~|~i=1,2,\cdots, n,~q\in \mathbb{H}\}$ a quaternionic rank $n$ continuous Riesz family if there exists two constants $A,B>0$ such that for every  scalar sequence $\{c_{i}\}_{i=1}^{n}\subseteq\quat^n,$ 
	\begin{equation*}
	A\sum_{i=1}^{n}|c_{i}|^{2}\leq\left\|\sum_{i=1}^{n}\int_{\mathbb{H}}\eta^i_q c_{i} d\mu(q)\right\|^{2}\leq B\sum_{i=1}^{n}|c_{i}|^{2},
	\end{equation*}
	where $A,B$ are called Riesz bounds. 
\end{definition}
\begin{theorem}
	Let $\{\eta^i_q\in V_\mathbb{H}^R~|~i=1,2,\cdots, n,~q\in \mathbb{H}\}$ be a quaternionic rank $n$ continuous Riesz family in $V_{\quat}^{R}$ with bounds $m(A)$ and $M(A)$ and let $\{\xi^i_q\in V_\mathbb{H}^R~|~i=1,2,\cdots, n,~q\in \mathbb{H}\}$ be a family in $V_\mathbb{H}^R$ which satisfies $\gamma=\displaystyle\sum_{i=1}^{n}\int_{\mathbb{H}}\left\|\eta^i_q-\xi^i_q\right\| \left\| S^{-1}\eta^i_q\right\| d\mu(q)<1. $ Then $\{\xi^i_q\in V_\mathbb{H}^R~|~i=1,2,\cdots, n,~q\in \mathbb{H}\}$  is a rank $n$ continuous Riesz family with bounds $m(A)(1-\gamma^{2})$ and  $M(A)\left(1+\sqrt{\displaystyle\frac{\lambda}{M(A)}} \right)^{2} $, where $\lambda:=\displaystyle\sum_{i=1}^{n}\int_{\mathbb{H}}\left\|\eta^i_q-\xi^i_q\right\|^{2}d\mu(q)$ and $S$ is a frame operator of  $\{\eta^i_q\in V_\mathbb{H}^R~|~i=1,2,\cdots, n,~q\in \mathbb{H}\}$ in $L:=\overline{ rightspan}\{\eta^i_q~|~i=1,2,..,n;~q\in\quat\}.$
\end{theorem}
\begin{proof}
	For $\{c_{i}\}_{i=1}^{n}\in\quat^{n},$
	\begin{eqnarray*}
		\left\|\sum_{i=1}^{n}\int_{\mathbb{H}} \xi^i_q c_{i} d\mu(q)\right\|&\leq&	
		\sum_{i=1}^{n}\int_{\mathbb{H}}\left\| \xi^i_q c_{i}\right\| d\mu(q)\\&=&\sum_{i=1}^{n}\int_{\mathbb{H}}\left\| (\xi^i_q-\eta^i_q+\eta^i_q)c_i\right\| d\mu(q)\\
		&\leq&\sum_{i=1}^{n}\int_{\mathbb{H}}\left\|(\eta^i_q-\xi^i_q)c_i \right\|d\mu(q)+\sum_{i=1}^{n}\int_{\mathbb{H}}\left\| \eta^i_qc_i \right\|d\mu(q)\\
		&=&  \sum_{i=1}^{n}\left|c_{i}\right|  \int_{\mathbb{H}}\left\|(\eta^i_q-\xi^i_q) \right\|d\mu(q)+\sum_{i=1}^{n}\int_{\mathbb{H}}\left\| \eta^i_qc_i \right\|d\mu(q)\\
		&\leq&\left( \sum_{i=1}^{n}\left|c_{i}\right|^{2}\right)^{\frac{1}{2}} \left(\sum_{i=1}^{n}\int_{\mathbb{H}}\left\|(\eta^i_q-\xi^i_q) \right\|^{2}d\mu(q) \right)^{\frac{1}{2}}+ \sqrt{M(A)}\left( \sum_{i=1}^{n}\left|c_{i}\right|^{2}\right)^{\frac{1}{2}}\\
		&=&\sqrt{\lambda}\left( \sum_{i=1}^{n}\left|c_{i}\right|^{2}\right)^{\frac{1}{2}}+\sqrt{M(A)}\left( \sum_{i=1}^{n}\left|c_{i}\right|^{2}\right)^{\frac{1}{2}}\\
		&=&\sqrt{M(A)}\left(1+\sqrt{\frac{\lambda}{M(A)}} \right) \left( \sum_{i=1}^{n}\left|c_{i}\right|^{2}\right)^{\frac{1}{2}}.
	\end{eqnarray*}
	Define $U:V_{\quat}^{R}\longrightarrow V_{\quat}^{R},\mbox{~by~} U\phi=\displaystyle\sum_{i=1}^{n}\int_{\mathbb{H}} \xi^i_q \left\langle \phi|S^{-1}\eta^i_q\right\rangle d\mu(q).$ For any $\phi\in V_\quat^R,$ we have
	\begin{eqnarray*}
		\left\| \sum_{i=1}^{n}\int_{\mathbb{H}} \xi^i_q \left\langle \phi|S^{-1}\eta^i_q\right\rangle d\mu(q)\right\| &\leq& \sum_{i=1}^{n}\int_{\mathbb{H}}\left\|  \xi^i_q \left\langle \phi|S^{-1}\eta^i_q\right\rangle \right\| d\mu(q)\\
		&\leq&(\sqrt{\lambda}+\sqrt{M(A)})\left(\sum_{i=1}^{n}\int_{\mathbb{H}}\left|\left\langle \phi|S^{-1}\eta^i_q\right\rangle \right|^{2}  \right)^{\frac{1}{2}} 
		\end{eqnarray*}
		\begin{eqnarray*}
		&=&(\sqrt{\lambda}+\sqrt{M(A)})\left(\sum_{i=1}^{n}\int_{\mathbb{H}}\left|\left\langle P_{L}\phi|S^{-1}\eta^i_q\right\rangle \right|^{2}  \right)^{\frac{1}{2}} \\
		&\leq&(\sqrt{\lambda}+\sqrt{M(A)})\left( \frac{1}{m(A)}\left\| P_{L}\phi\right\|^{2} \right)^{\frac{1}{2}} \\
		&\leq&\frac{\sqrt{\lambda}+\sqrt{M(A)}}{\sqrt{m(A)}}\left\| \phi\right\|. 
	\end{eqnarray*}
	Since $\{\eta^i_q\in V_\mathbb{H}^R~|~i=1,2,\cdots, n,~q\in \mathbb{H}\}$ is a continuous frame for $L$, by the frame decomposition, $U(\eta^i_q)=\xi^i_q,$ for all $i=1,2,..,n$ and $q\in\quat$.
	For $\phi\in L,$
	\begin{eqnarray*}
		\left\| \phi-U\phi\right\| &=&\left\|\sum_{i=1}^{n}\int_{\mathbb{H}} \eta^i_q \left\langle \phi|S^{-1}\eta^i_q\right\rangle d\mu(q)-\sum_{i=1}^{n}\int_{\mathbb{H}} \xi^i_q \left\langle \phi|S^{-1}\eta^i_q\right\rangle d\mu(q) \right\| \\
		&=&\left\| \sum_{i=1}^{n}\int_{\mathbb{H}}(\eta^i_q -\xi^i_q )\left\langle \phi|S^{-1}\eta^i_q\right\rangle d\mu(q)\right\| \\
		&\leq& \sum_{i=1}^{n}\int_{\mathbb{H}}\left\|(\eta^i_q -\xi^i_q )\left\langle \phi|S^{-1}\eta^i_q\right\rangle\right\| d\mu(q) \\
		&=&\sum_{i=1}^{n}\int_{\mathbb{H}}\left\|\eta^i_q -\xi^i_q \right\| \left|\left\langle \phi|S^{-1}\eta^i_q\right\rangle \right|d\mu(q) \\
		&\leq&\sum_{i=1}^{n}\int_{\mathbb{H}}\left\|\eta^i_q -\xi^i_q \right\| \left\| S^{-1}\eta^i_q\right\| \left\| \phi\right\|d\mu(q)\\
		&=& \gamma\left\| \phi\right\|.
	\end{eqnarray*}
	Hence $\left\| \phi-U\phi\right\|\leq \gamma\left\| \phi\right\|.$ It follows that $\left|\left\|\phi \right\| -\left\| U\phi\right\| \right| \leq \left\| \phi-U\phi\right\| \leq\gamma\left\| \phi\right\| $ and $\left\| U\phi\right\| \geq(1-\gamma)\left\| \phi\right\|.$
	We have
	\begin{eqnarray*}
		\left\|\sum_{i=1}^{n}\int_{\mathbb{H}} \xi^i_q c_{i} d\mu(q)\right\|&=&\left\|U\left(\sum_{i=1}^{n}\int_{\mathbb{H}} \eta^i_q c_{i} d\mu(q) \right)  \right\| \\
		&\geq&(1-\gamma)\left\| \sum_{i=1}^{n}\int_{\mathbb{H}} \eta^i_q c_{i} d\mu(q)\right\| \\
		&\geq&(1-\gamma)\sqrt{m(A)}\left( \sum_{i=1}^{n}\left| c_{i}\right|^{2}\right)^{\frac{1}{2}}. 
	\end{eqnarray*}
\end{proof}
\section{acknowledgment}
K. Thirulogasanthar would like to thank the, FRQNT, Fonds de la Recherche  Nature et  Technologies (Quebec, Canada) for partial financial support under the grant number 2017-CO-201915. Part of this work was done while he was visiting the University of Jaffna, Sri Lanka. He expresses his thanks for the hospitality.

\end{document}